\documentclass{birkjour}
\usepackage{url}
 \newtheorem{thm}{Theorem}[section]
 \newtheorem{cor}[thm]{Corollary}
 \newtheorem{lem}[thm]{Lemma}
 
 \theoremstyle{definition}
 
 \theoremstyle{remark}

 \numberwithin{equation}{section}

\begin{document}

%
%
%
%
%
%
%
%
%

\title[Sturm-Liouville problems]
{Sturm-Liouville problems with a boundary condition depending linearly on an eigenparameter}

\author[Aliyev]{Yagub N. Aliyev}

\address{ADA University\\
School of IT and Engineering\\
Ahmadbay Agha-Oglu Street~61\\
Baku AZ1008, Azerbaijan\\}

\email{yaliyev@ada.edu.az}

\author[Aliyeva]{Narmin N. Aliyeva}

\address{Baku State University\\
Department of Differential and Integral Equations\\
Academic Zahid Khalilov str.~23\\
Baku AZ1148, Azerbaijan\\}

\email{nermin.aliyeva@idrak.edu.az}

\thanks{This work was completed with the support of ADA University Faculty Research and Developement Fund and Baku State University.}



\date{October 11, 2019}


\begin{abstract}
This paper studies a Sturm--Liouville boundary value problem in which one of the boundary conditions depends linearly on the spectral parameter. The differential equation is considered on the interval $(0,1)$ with a classical boundary condition at one endpoint and an eigenparameter--dependent boundary condition at the other.
Explicit formulas for the inner products and norms of the root functions are obtained. These relations make it possible to analyze the structure of the system of root functions and the corresponding biorthogonal system. Using these results, the minimality of the system of root functions in $L_2(0,1)$ is established.
Furthermore, the basis properties of the system of root functions in the spaces $L_p(0,1)$, $1<p<\infty$, are investigated. Necessary and sufficient conditions under which the system forms a basis are derived. Special attention is given to the cases of multiple eigenvalues and the case when the eigenvalue coincides with the critical value $-d/c$. The obtained results reveal a symmetry between different spectral cases and provide a simpler approach that avoids the use of the exit space $L_2(0,1) \oplus \mathbb{C}$. Several examples are presented to illustrate the theoretical results.
\end{abstract}

\maketitle
\section{Introduction}

Consider the following spectral problem
$$
-y^{\prime \prime }+q(x)y=\lambda y,\ 0<x<1, \eqno(1.1)
$$
$$
y(0)\cos \beta =y^{\prime }(0)\sin \beta ,\ 0\leq \beta <\pi ,
\eqno(1.2)
$$
$$
(a\lambda + b)y(1) = (c\lambda + d)y{'}(1), \eqno(1.3)
$$
where $a,b,c,d$ are real constants, $ad-bc<0$, $\lambda $ is the
spectral parameter, and $q(x)$ is a real-valued and continuous function over the interval $[0,1]$. 
The case \(ad-bc>0\) was studied in \cite{fulton,kerimov1}.
The cases \(a=0\) and \(c=0\) were studied in \cite{aliyev1} and \cite{aliyev2}, respectively. Therefore, throughout this paper we will assume that $ad-bc<0$ and \(ac\neq0\). The case $ad-bc<0$ was also studied in \cite{zaliyev1} and their approach used the introduction of the exit space $L_2\oplus \mathbb{C}$ with an indefinite metric. This method and the solution of more general problem about the basis properties were discussed in \cite{rus} and \cite{shkalikov1}. The special problem similar to (1.1)-(1.3) but with the condition (1.2) replaced by more special $y'(0)=0$ and other related problems, were studied in \cite{shkalikov2} as examples for this method. The current paper does not use the exit space $L_2\oplus \mathbb{C}$ and its indefinite metric. In \cite{kerimov0} a more general problem with eigenvalue parameter both in (1.2) and (1.3) was studied. In \cite{kerimov0} only sufficient conditions for the basis properties were obtained.

It was shown in \cite{binding1} that the eigenvalues of the boundary value problem (1.1)--(1.3) form an infinite sequence with the only accumulation point at $+\infty$. Moreover, exactly one of the following alternatives occurs:
\textbf{(i)} all eigenvalues are real and simple;
\textbf{(ii)} all eigenvalues are real and, except for a single eigenvalue of algebraic multiplicity 2, are simple;
\textbf{(iii)} all eigenvalues are real and, except for a single eigenvalue of algebraic multiplicity 3, are simple;
\textbf{(iv)} all eigenvalues are simple, and apart from one conjugate pair of non-real eigenvalues, they are all real.

In the current paper, we study the basis properties in $L_{p}(0,1)$ $(1<p<+\infty)$ of the system of root functions corresponding to the boundary value problem (1.1)--(1.3). For this we first prove the minimality of the system in $L_{2}(0,1)$. Note that the system $\{v_n\}_{n=0}^\infty \subset L_2(0,1)$ is called
\emph{minimal} if for every $k=0,1,2,\ldots$
\[
v_k \notin \overline{\operatorname{span}}\{v_n : n \neq k\}.
\]
The eigenvalues $\{\lambda_n\}_{n\ge 0}$ are arranged in order of
non-decreasing real parts and counted with their algebraic
multiplicities. Their asymptotic behavior is given by
\cite[Theorem~2.2]{binding2} as
$$
\lambda_n =
\begin{cases}
\left(n-1\right)^2\pi^2 + O(1), & \text{if } \beta \neq 0,\\[6pt]
\left(n-\dfrac{1}{2}\right)^2\pi^2 + O(1), & \text{if } \beta = 0.
\end{cases}\eqno(1.4)
$$
Using (1.4) one can show that the root functions with one function removed form a system quadratically close to the trigonometric systems except some special cases which we focus on in this paper. This method of using the quadratic closeness and Bari's theorem for the completeness in $L_2(0,1)$ first appeared in the paper of A. G. Kostyuchenko and A. V. Skorokhod \cite{kost}. The method of using the biorthogonal system for the minimality is due to E.I. Moiseev and N.Yu. Kapustin (see \cite{mois} and its references).

The main novelty of the results in the current paper is more general boundary condition (1.2) (cf. \cite{shkalikov2} where $y'(0)=0$), and that all the choices of the removed root function are considered case by case. Also, the paper reveals a symmetry between the cases \(\lambda_{n}= -\frac{d}{c}\) and \(\lambda_{n}\neq -\frac{d}{c}\), which was not observed in the earlier works (cf. \cite{kerimov0}, \cite{kerimov1}). Another important feature of the paper is that we did not use the characteristic function 
\[\omega(\lambda) = (a\lambda + b)y(1,\lambda) - (c\lambda + d)y{'}(1,\lambda),\]
in the proofs. The earlier works heavily relied on the use of this function. Avoiding the use of the characteristic function made the proofs below more simple but it required more delicate use of the associated functions.

The obtained results are in perfect agreement with Theorem 3 of \cite{shkalikov2}. In particular, it was mentioned by A. A. Shkalikov \cite{shkalikov2} that in the case of an eigenvalue $\lambda_k$ of multiplicity 2, the system of root functions with eigenfunction $y_k$ removed is not always a basis in $L_2(0,1)$, and if the associated function $y_{k+1}$ is removed, then it is always a basis. Similarly, in the case of an eigenvalue $\lambda_k$ of multiplicity 3, the system of root functions with eigenfunction $y_k$ or associated function $y_{k+1}$ removed is not always a basis in $L_2(0,1)$, and if the associated function $y_{k+2}$ is removed, then it is always a basis. The current paper uses special associated functions denoted by $y_{k+1}^*$, $y_{k+1}^\#$, and $y_{k+2}^\#$ to write the necessary and sufficient conditions for the basis properties in explicit form, which was missing in the literature. Several examples demonstrate the method in special boundary value problems.

\section{Inner products and norms of eigenfunctions.}
By (1.1)-(1.3), if $y_{n}$ is an eigenfunction corresponding to $\lambda_n$, then
$$
-y_{n}^{\prime \prime }+q(x)y_{n}=\lambda _{n}y_{n},  \eqno(2.1)
$$
$$
y_{n}^{\prime }(0)\sin \beta =y_{n}(0)\cos \beta ,  \eqno(2.2)
$$
$$
(a\lambda _{n}+b)y_{n}(1)=(c\lambda _{n}+d)y^{\prime }_{n}(1).  \eqno(2.3)
$$
As usual $(\cdot ,\cdot )$ and  ${\left\| \cdot\right\|}_{2}$ mean the inner
product and the norm, respectively, in $L_{2}(0,1)$.

\begin{lem}
 Let $y_{n}, y_{m}$ be eigenfunctions
corresponding to the eigenvalues $\lambda _{n}$ and $\lambda _{m}$
such that $\lambda _{n}\ne \overline{\lambda _{m}}$.

\noindent (a) If \(\lambda_{n},\lambda_{m}\neq - \frac{d}{c}\), then 

\[\left( y_{n},y_{m} \right) = - (ad-bc)\frac{y_{n}(1)\overline{y_{m}(1)}}{\left( c\lambda_{n} + d \right)\left( c{\overline{\lambda}}_{m} + d \right)}.\eqno(2.4)\]

\noindent (b) If \(\lambda_{n}= -\frac{d}{c} \neq \lambda_{m}\), then 
\[\left( y_{n},y_{m} \right) =  -(ad-bc)\frac{{y'_{n}}(1)\overline{y_{m}(1)}}{(a\lambda_{n} + b)(c\overline{\lambda_{m}}+d)}.\eqno(2.5)\]
\end{lem}

\begin{proof} By Lagrange's identity
$$
\frac{d}{dx}(y_n(x)\overline{y_m^{\prime }(x)}-y_n^{\prime
}(x)\overline{y_m(x)})= (\lambda_n - \overline
{\lambda_m})y_n(x)\overline{y_m(x)}.
$$
By integrating both sides of this equality from 0 to 1, we obtain
$$
(\lambda_n - \overline
{\lambda_m})(y_n,y_m
)={\left. {(y_n(x)\overline{y_m^{\prime }(x)}-y_n^{\prime
}(x)\overline{y_m(x)}) } \right|}_{0}^{1}.
$$
From (1.2), we obtain
$$
y_n(0)\overline{y_m^{\prime }(0)}-y_n^{\prime }(0
)\overline{y_m(0)}=0. \eqno(2.6)
$$
Therefore, if \(\lambda_{n},\lambda_{m}\neq - \frac{d}{c}\), then
$$
(\lambda_n - \overline
{\lambda_m})(y_n,y_m
)=y_n(1)\overline{y_m^{\prime }(1)}-y_n^{\prime }(1
)\overline{y_m(1)}
$$
$$
=- \frac{(\lambda_n - \overline
{\lambda_m})y_{n}(1)\overline{y_{m}(1)}(ad - bc)}{\left( c\lambda_{n} + d \right)\left( c{\overline{\lambda}}_{m} + d \right)}. \eqno(2.7)
$$
By cancelling $\lambda_n - \overline
{\lambda_m}$ from both sides of (2.7) we obtain (2.4). If \(\lambda_{n}= - \frac{d}{c} \neq \lambda_{m}\), then $y_n(1)=0$ and
$$
(\lambda_n - \overline{\lambda_m})(y_n,y_m)=-y_n^{\prime }(1)\overline{y_m(1)}.\eqno(2.8)
$$
By dividing (2.8) by $\lambda_n - \overline{\lambda_m}=-\frac{c{\overline{\lambda}}_{m} + d }{c}$, we obtain (2.5).
\end{proof}

\begin{cor}
If $\lambda _{r}$ is a non-real eigenvalue, then
    \[{\left\| {y_{r} }\right\|}^{2}_{2} = -(ad - bc) \frac{\left| {y_{r}(1)} \right|^2}{\left| c\lambda_{r} + d \right|^2}.\eqno(2.9)\]
\end{cor}
\begin{proof}
Since $\lambda _{r}$ is a non-real eigenvalue, $\lambda_{r}\ne {\overline{\lambda}}_{r}$, and by (2.4),
    \[\left( y_{r},y_{r} \right) = -(ad - bc) \frac{y_{r}(1)\overline{y_{r}(1)}}{\left( c\lambda_{r} + d \right)\left( c{\overline{\lambda}}_{r} + d \right)},\]
which is equivalent to (2.9).
\end{proof}
If \(\lambda_{k}\) is a double eigenvalue
(\(\lambda_{k} = \lambda_{k + 1}\)) or if \(\lambda_{k}\) is a triple eigenvalue
(\(\lambda_{k} = \lambda_{k + 1} = \lambda_{k + 2}\)), then for the associated function
\(y_{k + 1}\) corresponding to the eigenfunction \(y_{k}\), the
following relations hold:
\[{- y''_{k + 1}} + q(x)y_{k + 1} = \lambda_{k}y_{k + 1} + y_{k}, \eqno(2.10)\]
\[y_{k + 1}(0)\cos\beta = y'_{k + 1}(0)\sin\beta, \eqno(2.11)\]
\[\left( a\lambda_{k} + b \right)y_{k + 1}(1) + ay_{k}(1) = \left( c\lambda_{k} + d \right)y'_{k + 1}(1) + cy'_{k}(1). \eqno(2.12)\]

\begin{lem} Suppose that $\lambda _{k}$ is a multiple eigenvalue \((\lambda_{k}=\lambda_{k+1} \le\lambda_{k+2})\).

\noindent(a) If \(\lambda_{k}\neq - \frac{d}{c}\), then
$$
{\left\| {y_{k} } \right\|}^{2}_{2}=-(ad-bc)\frac{y_{k}^2(1)}{( c\lambda_{k} + d )^2}.
\eqno(2.13)
$$
(b) If \(\lambda_{k}= - \frac{d}{c}\), then
$$
{\left\| {y_{k} } \right\|}^{2}_{2}=-(ad-bc)\frac{(y_{k}^{\prime}(1))^2}{ (a\lambda_{k} + b)^2}
\eqno(2.14)
$$
\end{lem}
\begin{proof}
  By Lagrange's identity and (2.10),
  $$\frac{d}{dx}(y_{k+1}(x){y_{k}^{\prime }(x)}-y_k(x){y_{k+1}^{\prime }(x)})={y_{k}^{2}(x)} .$$
By integrating both sides of this identity from 0 to 1, and using (2.11), 
 $$\left\| y_{k}\right\|^{2}_{2}={(y_{k+1}(x){y_{k}^{\prime }(x)}-y_k(x){y_{k+1}^{\prime }(x)} |}_{0}^{1}=y_{k+1}(1){y_{k}^{\prime }(1)}-y_k(1){y_{k+1}^{\prime }(1)}.
$$
If \(\lambda_{k}\neq - \frac{d}{c}\), then by (2.12),
 $$\left\| y_{k}\right\|^{2}_{2}=y_{k+1}(1)\frac{{y_{k}(1)}( a\lambda_{k} + b)}{( c\lambda_{k} + d)}-y_k(1)\frac{{y_{k+1}(1)}( a\lambda_{k} + b)+ay_{k}(1)-cy_{k}{\prime}(1)}{c\lambda_{k} + d}.
$$
By simplifying further we obtain (2.13).
If \(\lambda_{k}= - \frac{d}{c}\), then by (2.3), $y_k(1)=0$, and by (2.12),
$$
{y_{k+1}(1)}=\frac{cy_{k}^{\prime}(1)}{ a\lambda_{k} + b}.\eqno(2.15)
$$
Therefore $\left\| y_{k}\right\|^{2}_{2}=\frac{c(y^{\prime}_{k}(1))^2}{ a\lambda_{k} + b},$
which proves (2.14).
\end{proof}
Let
\[
B_n =
\begin{cases}
\displaystyle \|y_n\|_{2}^{2} 
+ (ad - bc)\dfrac{|y_n(1)|^{2}}{|c\lambda_n + d|^2}, 
& \text{if }  \lambda_n \neq -\dfrac{d}{c}, \\[10pt]
\displaystyle \|y_n\|_{2}^{2} 
+(ad-bc)\dfrac{\bigl(y_n'(1)\bigr)^{2}}{(a\lambda_n + b)^2}, 
& \text{if }  \lambda_n = -\dfrac{d}{c}.
\end{cases}\eqno(2.16)
\]

\begin{lem}
$B_n \neq 0$ if and only if the corresponding eigenvalue $\lambda_n$ is real and has multiplicity one (is simple).
\end{lem}
\begin{proof}
    By (2.9), if $\lambda _{r}$ is a non-real eigenvalue, then $B_r=0$. Similarly, if $\lambda _{k}$ is a multiple eigenvalue, then by (2.13) and (2.14), $B_k=0$. Suppose now that $\lambda _{n}$ is real and simple.
Let $Y(x)$ be a function satisfying
\[- Y'' + q(x)Y = \lambda_{n}Y + y_{n}, \eqno(2.17)\]
\[Y(0)\cos\beta = Y'(0)\sin\beta. \eqno(2.18)\]
Since $\lambda _{n}$ is simple,
\[\left( a\lambda_{n} + b \right)Y(1) + ay_{n}(1)\ne \left( c\lambda_{n} + d \right)Y'(1) + cy'_{n}(1). \eqno(2.19)\]
By (2.17), 
$$\frac{d}{dx}(Y{y_{n}^{\prime }}-y_n{Y^{\prime }})={y_{n}^{2}} .$$
By integrating both sides of this equality from 0 to 1, and then using (2.18) and (2.19), we obtain 
\[
 \|y_n\|_{2}^{2}\ne
\begin{cases}
\displaystyle  
-\dfrac{y^{2}_n(1)(ad - bc)}{(c\lambda_n + d)^2}, 
& \text{if }  \lambda_n \neq -\dfrac{d}{c}, \\[10pt]
\displaystyle 
 \dfrac{c\bigl(y_n'(1)\bigr)^{2}}{a\lambda_n + b}, 
& \text{if }  \lambda_n = -\dfrac{d}{c},
\end{cases}
\]
which completes the proof.
\end{proof}

The following result is proved in the same way.

\begin{lem}
If $\lambda_r$ and $\lambda_s = \overline{\lambda_r}$ ($s:=r+1$) form a conjugate pair of non-real eigenvalues, then

$$
T_s= \overline{T_r}:=(y_{s},y_{r})+(ad-bc)\frac{y_{s}^2(1)}{( c\lambda_{s} + d )^2}\ne 0.\eqno(2.20)
$$ 
\end{lem}

\begin{proof}
Since $\lambda _{r}$ is a simple eigenvalue, any function $Y(x)$ satisfying
\[- Y'' + q(x)Y = \lambda_{r}Y + y_{r}, \eqno(2.21)\]
\[Y(0)\cos\beta = Y'(0)\sin\beta, \eqno(2.22)\]
also satisfies
\[\left( a\lambda_{r} + b \right)Y(1) + ay_{r}(1)\ne \left( c\lambda_{r} + d \right)Y'(1) + cy'_{r}(1). \eqno(2.23)\]
By (2.21), 
$$\frac{d}{dx}{y_s^{\prime }\overline{Y}}-{\overline{Y^{\prime }}{y_{s}}}={y}_s\overline{y_{r}}.$$
By integrating both sides of this equality from 0 to 1, and then using (2.22) and (2.23), we obtain 
$$(y_{s},y_{r})\neq-(ad-bc)\frac{y_{s}^2(1)}{( c\lambda_{s} + d )^2},$$
which proves (2.20). Note that since $\lambda_{s}$ is not a real number, $c\lambda_{s} + d\ne 0$.
\end{proof}

\section{Inner products with the first associated function.}

\begin{lem} Suppose that \(\lambda_{k}\) is an eigenvalue of multiplicity two or three \((\lambda_{k}=\lambda_{k+1}\le\lambda_{k+2})\).

\noindent(a) If \(\lambda_{k}\neq - \frac{d}{c}\neq\lambda_n\), then
$$
\left( y_{k + 1},y_{n} \right) = - (ad - bc)\left( \frac{y_{k + 1}(1)}{c\lambda_{k} + d} - \frac{cy_{k}(1)}{\left( c\lambda_{k} + d \right)^{2}}\right)\cdot \frac{y_{n}(1)}{ c\lambda_{n} + d}.\eqno(3.1)
$$
(b) If \(\lambda_{k}\neq - \frac{d}{c}=\lambda_n\), then
\[\left( y_{k + 1},y_{n} \right) = - (ad - bc)\left( \frac{y_{k + 1}(1)}{c\lambda_{k} + d} - \frac{cy_{k}(1)}{\left( c\lambda_{k} + d \right)^{2}}\right)\cdot \frac{y'_{n}(1)}{ a\lambda_{n} + b}.\eqno(3.2)\]
(c) If \(\lambda_{k} = - \frac{d}{c}\neq\lambda_n\), then
\[\left( y_{k+1},y_{n} \right) =-(ad-bc)\left(\frac{y'_{k + 1}(1)}{a\lambda_{k}+b}-\frac{ay'_{k}(1)}{(a\lambda_{k}+b)^2}\right)\cdot\frac{y_{n}(1)}{ c\lambda_{n} + d}.\eqno(3.3)\]
\end{lem}

\begin{proof}
We start with the identity
$$
\frac{d}{dx}\left(y_{k+1}y^{\prime}_{n}-y^{\prime}_{k+1}y_{n}\right)=(\lambda_k-\lambda _n)y_{k+1}y_{n}+y_{k}y_{n}.
$$
By integrating this equality from $0$ to $1$, we obtain
$$
(\lambda_k-\lambda _n)(y_{k+1},y_{n})+(y_{k},y_{n})={\left. \left(y_{k+1}y^{\prime}_{n}-y^{\prime}_{k+1}y_{n}\right) \right|}_{0}^{1}.\eqno(3.4)
$$
By (2.2) and (2.11), $y_{k+1}(0)y^{\prime}_{n}(0)-y^{\prime}_{k+1}(0)y_{n}(0)=0$.
By Lemma 2.1,
$$
(y_{k},y_{n})=
\begin{cases}
\displaystyle -(ad - bc)\dfrac{y_k(1)y_n(1)}{(c\lambda_k + d)(c\lambda_n + d)}, 
& \text{if }  \lambda_k \neq -\dfrac{d}{c}\neq\lambda_n, \\[10pt]
\displaystyle \dfrac{cy_k(1)y_n'(1)}{c\lambda_k + d}, 
& \text{if }  \lambda_k \neq -\dfrac{d}{c}=\lambda_n,
\\[10pt]
 
\displaystyle \dfrac{cy_k'(1)y_n(1)}{c\lambda_n + d}, 
& \text{if }  \lambda_k = -\dfrac{d}{c}\neq\lambda_n.
\end{cases} \eqno(3.5)
$$
Therefore, if $\lambda_k \neq -\dfrac{d}{c}\neq\lambda_n$, then by (2.3) and (2.12),
$$(\lambda_k-\lambda _n)(y_{k+1},y_{n})=
y_{k+1}(1)y^{\prime}_{n}(1)-y^{\prime}_{k+1}(1)y_{n}(1)-(y_{k},y_{n})
$$
$$
= y_{k+1}(1)\frac{y_{n}(1)(a\lambda_{n} + b)}{\left( c\lambda_{n} + d \right)}-\frac{{y_{k+1}(1)}( a\lambda_{k} + b)+ay_{k}(1)-cy_{k}{\prime}(1)}{c\lambda_{k} + d}y_{n}(1)-(y_{k},y_{n}).
$$
Using (3.5), dividing by $\lambda_k-\lambda _n$, and simplifying we obtain (3.1).
Similarly, if $\lambda_k \neq -\dfrac{d}{c}=\lambda_n$, then by (2.3), $y_n(1)=0$, and therefore
$$(\lambda_k-\lambda _n)(y_{k+1},y_{n})=
y_{k+1}(1)y^{\prime}_{n}(1)-(y_{k},y_{n}).
$$
Again, using (3.5), dividing by $\lambda_k-\lambda _n=\lambda_k+\dfrac{d}{c}$, and simplifying we obtain (3.2).
Finally, if $\lambda_k = -\dfrac{d}{c}\neq\lambda_n$, then by (2.15),
$$(\lambda_k-\lambda _n)(y_{k+1},y_{n})=
\frac{cy_{k}^{\prime}(1)}{ a\lambda_{k} + b}\cdot\frac{y_{n}(1)(a\lambda_{n} + b)}{\ c\lambda_{n} + d }-y^{\prime}_{k+1}(1)y_{n}(1)-(y_{k},y_{n}).
$$
Using (3.5) one more time, dividing by $\lambda_k-\lambda _n=-\dfrac{d}{c}-\lambda_n$, and simplifying we obtain (3.3).

\end{proof}

\begin{lem} Suppose that  $\lambda _{k}$ is an eigenvalue of multiplicity two  and not three  \((\lambda_{k}=\lambda_{k+1} <\lambda_{k+2})\).

\noindent(a) If $\lambda _{k}\ne -\frac{d}{c}$, then
$$
T_k:=(y_{k+1},y_{k})+(ad-bc)\left( \frac{y_{k + 1}(1)}{c\lambda_{k}+d} - \frac{cy_{k}(1)}{{(c\lambda_{k}+d)}^{2}}\right)\cdot \frac{y_{k}(1)}{c\lambda_{k}+d}\ne 0. \eqno(3.6)
$$
(b) If $\lambda _{k}=-\frac{d}{c}$, then 
$$S_k:=(y_{k+1},y_{k})+(ad-bc)\left(\frac{y'_{k + 1}(1)}{a\lambda_{k}+b}-\frac{ay'_{k}(1)}{(a\lambda_{k}+b)^2}\right)\cdot \frac {y'_{k}(1)}{a\lambda_{k}+b}
\ne 0. \eqno(3.7)$$
\end{lem}

\begin{proof}
(a) Since $\lambda _{k}$ is a double
eigenvalue and not a triple eigenvalue (\(\lambda_{k}=\lambda_{k+1} <\lambda_{k+2}\)),
the function $Y(x)$ satisfying
$$
-Y^{\prime \prime }+q(x)Y=\lambda _{k}Y+y_{k+1}, \eqno(3.8)
$$
$$
Y(0)\cos \beta  =Y^{\prime }(0)\sin \beta,     \eqno(3.9)
$$
should also satisfy the inequality 
$$
\left( a\lambda_{k} + b \right)Y(1) + ay_{k + 1}(1)\neq\left( c\lambda_{k} + d \right)Y'(1) + cy'_{k + 1}(1).  \eqno(3.10)
$$
By (2.1) and (3.8) we obtain
$$
\frac{d}{dx}\left(Yy^{\prime}_{k}-Y^{\prime}y_{k}\right)=y_{k+1}y_{k}.
$$
By integrating this equality from 0 to 1, and using the boundary conditions (2.2), (2.3), (3.9), (3.10) for  $y_{k}$ and $Y(x)$, we obtain 
$$
(y_{k+1},y_{k})={\left.\left(Yy^{\prime}_{k}-Y^{\prime}y_{k}\right) \right|}_{0}^{1} .\eqno(3.11)
$$
(a) If $\lambda _{k}\ne -\frac{d}{c}$, then by (2.3) and (3.10),
$$
(y_{k+1},y_{k})\neq -\frac{y_{k+1}(1)y_{k}(1)(ad-bc)}{( c\lambda_{k} + d )^2}+\frac{cy_{k}^2(1)(ad-bc)}{( c\lambda_{k} + d )^3}\
$$
which is equivalent to (3.6).

\noindent(b) If $\lambda _{k}=-\frac{d}{c}$, then by (2.3), (3.10), and (3.11),
$$
(y_{k+1},y_{k})\neq \frac{cy^\prime_{k+1}(1)y^\prime_{k}(1)}{ a\lambda_{k} + b }-\frac{ac(y^\prime_{k}(1))^2}{( a\lambda_{k} + b)^2}\
$$
which proves (3.7).
\end{proof}

\section{Inner products with the second associated function.}

If $\lambda_k$ is an eigenvalue of multiplicity three, i.e., 
$\lambda_k = \lambda_{k+1} = \lambda_{k+2}$, then in addition to the first-order 
associated function $y_{k+1}$ defined by (2.10)-(2.12), there exists a second-order 
associated function $y_{k+2}$:
$$
-y_{k+2}^{\prime \prime }+q(x)y_{k+2}=\lambda _{k}y_{k+2}+y_{k+1}, \eqno(4.1)
$$
$$
y_{k+2}(0)\cos \beta  =y_{k+2}^{\prime }(0)\sin \beta,     \eqno(4.2)
$$
$$
\left( a\lambda_{k} + b \right)y_{k + 2}(1) + ay_{k + 1}(1) = \left( c\lambda_{k} + d \right)y'_{k + 2}(1) + cy'_{k + 1}(1).  \eqno(4.3)
$$
\begin{lem}
Suppose that \(\lambda_{k}\) is an eigenvalue of multiplicity three \((\lambda_{k}=\lambda_{k+1}=\lambda_{k+2})\).

\noindent(a) If \(\lambda_{k} \neq - \frac{d}{c}\neq\lambda_n\), then
$$
\left( y_{k + 2},y_{n} \right) = - (ad - bc)\left(\frac{y_{k + 2}(1)}{c\lambda_{k} + d} - \frac{cy_{k+1}(1)}{\left( c\lambda_{k} + d \right)^{2}}
+\frac{c^2y_{k}(1)}{\left( c\lambda_{k} + d \right)^{3}}\right)\cdot\frac{y_{n}(1)}{ c\lambda_{n} + d}.\eqno(4.4)
$$
(b) If \(\lambda_{k}\neq - \frac{d}{c}=\lambda_n\), then
\[\left( y_{k + 2},y_{n} \right) = - (ad - bc)\left(\frac{y_{k + 2}(1)}{c\lambda_{k} + d} -\frac{cy_{k+1}(1)}{\left( c\lambda_{k} + d \right)^{2}}
+\frac{c^2y_{k}(1)}{\left( c\lambda_{k} + d \right)^{3}}\right)
\cdot \frac{y'_{n}(1)}{ a\lambda_{n} + b}. \eqno(4.5)\]
(c) If \(\lambda_{k} = - \frac{d}{c}\neq\lambda_n\), then
\[\left( y_{k+2},y_{n} \right) =- (ad - bc)\left(\frac{y'_{k + 2}(1)}{a\lambda_{k} + b} - \frac{ay'_{k+1}(1)}{\left( a\lambda_{k} + b\right)^{2}}
+\frac{a^2y'_{k}(1)}{\left( a\lambda_{k} + b \right)^{3}}\right)\cdot\frac{y_{n}(1)}{ c\lambda_{n} + d}\eqno(4.6)\]
\end{lem}
\begin{proof}
By (2.1) and (4.1),
$$
\frac{d}{dx}\left(y_{k+2}y^{\prime}_{n}-y^{\prime}_{k+2}y_{n}\right)=(\lambda_k-\lambda _n)y_{k+2}y_{n}+y_{k+1}y_{n}.
$$
Integration of this equality from $0$ to $1$ gives
$$
(\lambda_k-\lambda _n)(y_{k+2},y_{n})+(y_{k+1},y_{n})={\left. \left(y_{k+2}y^{\prime}_{n}-y^{\prime}_{k+2}y_{n}\right) \right|}_{0}^{1}.
$$
By (2.2) and (4.2), $y_{k+2}(0)y^{\prime}_{n}(0)-y^{\prime}_{k+2}(0)y_{n}(0)=0$.
By using Lemma 3.1 for \(\left( y_{k + 1},y_{n} \right)\) and simplifying we obtain (4.4)-(4.6).
\end{proof}
\begin{lem} Suppose that  $\lambda _{k}$ is an eigenvalue of multiplicity three \((\lambda_{k}=\lambda_{k+1} =\lambda_{k+2})\).

\noindent(a) If $\lambda _{k}\ne -\frac{d}{c}$, then
$$
(y_{k+1},y_{k})=-(ad-bc)\left( \frac{y_{k + 1}(1)}{c\lambda_{k}+d} - \frac{cy_{k}(1)}{{(c\lambda_{k}+d)}^{2}}\right)\cdot \frac{y_{k}(1)}{c\lambda_{k}+d}. \eqno(4.7)
$$
(b) If $\lambda _{k}=-\frac{d}{c}$, then 
$$(y_{k+1},y_{k})=-(ad-bc)\left(\frac{y'_{k + 1}(1)}{a\lambda_{k}+b}-\frac{ay'_{k}(1)}{(a\lambda_{k}+b)^2}\right)\cdot \frac {y'_{k}(1)}{a\lambda_{k}+b}.\eqno(4.8)$$
\end{lem}

\begin{proof}
First note that by (2.1) and (4.1),
$$\frac{d}{dx}\left(y_{k+2}y^{\prime}_{k}-y^{\prime}_{k+2}y_{k}\right)=y_{k+1}y_{k}.$$ 
The integration of this equality from $0$ to $1$ gives
$$
(y_{k+1},y_{k})={\left. \left(y_{k+2}y^{\prime}_{k}-y^{\prime}_{k+2}y_{k}\right) \right|}_{0}^{1}.
$$
By (2.2) and (4.2), $y_{k+2}(0)y^{\prime}_{k}(0)-y^{\prime}_{k+2}(0)y_{k}(0)=0$.
By using (2.3), (4.3), and simplifying we obtain (4.7) and (4.8).
\end{proof}
\begin{lem} Suppose that  $\lambda _{k}$ is an eigenvalue of multiplicity three \((\lambda_{k}=\lambda_{k+1} =\lambda_{k+2})\).

\noindent(a) If $\lambda _{k}\ne -\frac{d}{c}$, then
$$
(y_{k+2},y_{k})=- (ad - bc)\left(\frac{y_{k + 2}(1)}{c\lambda_{k} + d} - \frac{cy_{k+1}(1)}{\left( c\lambda_{k} + d \right)^{2}}
+\frac{c^2y_{k}(1)}{\left( c\lambda_{k} + d \right)^{3}}\right)\cdot\frac{y_{k}(1)}{ c\lambda_{k} + d}+Q_{k}, \eqno(4.9)
$$
where
$$
Q_{k}=\|y_{k+1}\|^2_{2}+(ad-bc)\left( \frac{y_{k + 1}(1)}{c\lambda_{k}+d} - \frac{cy_{k}(1)}{{(c\lambda_{k}+d)}^{2}}\right)^2\ne 0. \eqno(4.10)
$$
(b) If $\lambda _{k}=-\frac{d}{c}$, then 
$$(y_{k+2},y_{k})=- (ad - bc)\left(\frac{y'_{k + 2}(1)}{a\lambda_{k} + b} - \frac{ay'_{k+1}(1)}{\left( a\lambda_{k} + b\right)^{2}}
+\frac{a^2y'_{k}(1)}{\left( a\lambda_{k} + b \right)^{3}}\right)\cdot\frac{y'_{k}(1)}{ a\lambda_{k} + b}+P_{k}.\eqno(4.11)$$
where
$$
P_{k}=\|y_{k+1}\|^2_{2}+(ad-bc)\left(\frac{y'_{k + 1}(1)}{a\lambda_{k}+b}-\frac{ay'_{k}(1)}{(a\lambda_{k}+b)^2}\right)^2\ne 0. \eqno(4.12)
$$
\end{lem}
\begin{proof}
(a) Since $\lambda _{k}$ is a triple
eigenvalue and not a quadruple eigenvalue (\(\lambda_{k}=\lambda_{k+1} =\lambda_{k+2}<\lambda_{k+3}\)),
the function $Y(x)$ satisfying
$$
-Y^{\prime \prime }+q(x)Y=\lambda _{k}Y+y_{k+2}, \eqno(4.13)
$$
$$
Y(0)\cos \beta  =Y^{\prime }(0)\sin \beta,     \eqno(4.14)
$$
should also satisfy the inequality 
$$
\left( a\lambda_{k} + b \right)Y(1) + ay_{k + 2}(1)\neq\left( c\lambda_{k} + d \right)Y'(1) + cy'_{k + 2}(1).  \eqno(4.15)
$$
By (2.1) and (4.13) we obtain
$$
\frac{d}{dx}\left(Yy^{\prime}_{k}-Y^{\prime}y_{k}\right)=y_{k+2}y_{k}.
$$
By integrating this equality from 0 to 1, and using the boundary conditions (2.2), (2.3), (4.14) for $y_{k}$ and $Y(x)$, and the inequality (4.15), we obtain 
$$
(y_{k+2},y_{k})={\left.\left(Yy^{\prime}_{k}-Y^{\prime}y_{k}\right) \right|}_{0}^{1} .\eqno(4.16)
$$
(a) If $\lambda _{k}\ne -\frac{d}{c}$, then by (2.3) and (4.15),
$$
(y_{k+2},y_{k})\neq -\frac{y_{k+2}(1)y_{k}(1)(ad-bc)}{( c\lambda_{k} + d )^2}
$$
$$
+\frac{cy_{k+1}(1)y_{k}(1)(ad-bc)}{( c\lambda_{k} + d )^3}-\frac{c^2y_{k}^2(1)(ad-bc)}{( c\lambda_{k} + d )^4},
$$
which is equivalent to (4.9) with $Q_{k}\neq0$.

\noindent(b) If $\lambda _{k}=-\frac{d}{c}$, then by (2.3), (4.15), and (4.16),
$$
(y_{k+2},y_{k})\neq \frac{cy^\prime_{k+2}(1)y^\prime_{k}(1)}{ a\lambda_{k} + b }-\frac{acy^\prime_{k+1}(1)y^\prime_{k}(1)}{( a\lambda_{k} + b)^2}+\frac{a^2c(y^\prime_{k}(1))^2}{( a\lambda_{k} + b)^3},
$$
which proves (4.11) with $P_{k}\neq0$. On the other hand, by (2.10) and (4.1),
$$
\frac{d}{dx}\left(y_{k+2}y^{\prime}_{k+1}-y^{\prime}_{k+2}y_{k+1}\right)=-y_{k+2}y_{k}+y^2_{k+1}.
$$ 
Integrating this from 0 to 1 gives
$$-(y_{k+2},y_{k})+\left\| y_{k+1}\right\|^{2}_{2}=\left(y_{k+2}y^{\prime}_{k+1}-y^{\prime}_{k+2}y_{k+1})\right|_{0}^{1}.$$
By (2.11) and (4.2), \(y_{k+2}(0)y^{\prime}_{k+1}(0)-y^{\prime}_{k+2}(0)y_{k+1}(0)=0\). By using (2.12), (4.3), and simplifying we obtain the equalities in (4.10) and (4.12). We will now prove the inequality $Q_{k}\ne 0$. One can check that the change of $y_{k+1}$ to $y_{k+1}^\dagger
:=y_{k+1}+\mu y_{k}$ does not change $Q_{k}$. Therefore, we can choose $\mu$ such that $Q_{k}=\|y_{k+1}^\dagger
\|^2_{2}$. Since $y_{k+1}$ and $y_{k}$ are linearly independent, $y_{k+1}^\dagger
\not\equiv0$, and therefore $Q_{k}\ne 0$. The proof of $P_{k}\ne 0$ is similar.
\end{proof}

\section{Existence of special associated functions $y_{k+1}^*$ and $y_{k+1}^\#$.} 
The following well-known properties of associated functions are crucial for our analysis. If $C$ and $D$ are arbitrary constants, then the functions $y_{k+1} + C y_k$ and $y_{k+2} + D y_k$ are also associated functions of the first and second orders, respectively. Furthermore, if we replace the associated function $y_{k+1}$ with $y_{k+1} + C y_k$, then the associated function $y_{k+2}$ is correspondingly transformed into $y_{k+2} + C y_{k+1}$.

\begin{lem}
Suppose that $\lambda_k\neq-\frac{d}{c}$ is an eigenvalue of multiplicity two. Then there exists an associated function
\[
y_{k+1}^* = y_{k+1} + C_1 y_k,
\]
with $C_1$ being a constant, such that
\[
(y_{k+1}^*, y_{k+1}) =-(ad-bc)\left(\frac{y^*_{k + 1}(1)}{c\lambda_{k}+d}-\frac{cy_{k}(1)}{(c\lambda_{k}+d)^2}\right)\left(\frac{y_{k + 1}(1)}{c\lambda_{k}+d}-\frac{cy_{k}(1)}{(c\lambda_{k}+d)^2}\right). \eqno(5.1)
\]

\end{lem}

\begin{proof} By (4.10),
\[
(y_{k+1}, y_{k+1}) =-(ad-bc)\left(\frac{y_{k + 1}(1)}{c\lambda_{k}+d}-\frac{cy_{k}(1)}{(c\lambda_{k}+d)^2}\right)^2+Q_{k}, \eqno(5.2)
\]
By (3.6),
$$
(y_{k}, y_{k+1})=-(ad-bc)\frac{y_{k}(1)}{c\lambda_{k}+d}\cdot \left( \frac{y_{k + 1}(1)}{c\lambda_{k}+d} - \frac{cy_{k}(1)}{{(c\lambda_{k}+d)}^{2}}\right)+T_{k}, \eqno (5.3)
$$
where $T_{k}\neq0$. By adding equations (5.2) and (5.3) multiplied by 
$C_1 = -\frac{Q_k}{T_k}$, we obtain
\[
\bigl(y_{k+1} + C_1 y_k, \, y_{k+1}\bigr) 
= -(ad-bc)\left(\frac{y_{k + 1}(1)+C_1y_k(1)}{c\lambda_{k}+d}-\frac{cy_{k}(1)}{(c\lambda_{k}+d)^2}\right)\times
\]
\[
\times\left(\frac{y_{k + 1}(1)}{c\lambda_{k}+d}-\frac{cy_{k}(1)}{(c\lambda_{k}+d)^2}\right),
\]
which proves (5.1).
\end{proof}

\begin{lem}
Suppose that $\lambda_k=-\frac{d}{c}$ is an eigenvalue of multiplicity two.  Then there exists an associated function
\[
y_{k+1}^* = y_{k+1} + C_1 y_k,
\]
with $C_1$ being a constant, such that
\[
(y_{k+1}^*, y_{k+1}) =-(ad-bc)\left(\frac{(y^*_{k + 1})'(1)}{a\lambda_{k}+b}-\frac{ay'_{k}(1)}{(a\lambda_{k}+b)^2}\right)\left(\frac{y'_{k + 1}(1)}{a\lambda_{k}+b}-\frac{ay'_{k}(1)}{(a\lambda_{k}+b)^2}\right). \eqno(5.4)
\]

\end{lem}

\begin{proof} By (4.12),
\[
(y_{k+1}, y_{k+1}) =-(ad-bc)\left(\frac{y'_{k + 1}(1)}{a\lambda_{k}+b}-\frac{ay'_{k}(1)}{(a\lambda_{k}+b)^2}\right)^2+P_{k}. \eqno(5.5)
\]
By (3.7),
$$
(y_{k}, y_{k+1})=-(ad-bc)\frac{y'_{k}(1)}{a\lambda_{k}+b}\cdot \left( \frac{y'_{k + 1}(1)}{a\lambda_{k}+b} - \frac{ay'_{k}(1)}{{(a\lambda_{k}+b)}^{2}}\right)+S_{k}, \eqno (5.6)
$$
where $S_{k}\neq0$. By adding equations (5.5) and (5.6) multiplied by 
$C_1 = -\frac{P_k}{S_k}$, we obtain
\[
\bigl(y_{k+1} + C_1 y_k, \, y_{k+1}\bigr) 
= -(ad-bc)\left(\frac{y'_{k + 1}(1)+C_1y'_k(1)}{a\lambda_{k}+b}-\frac{ay'_{k}(1)}{(a\lambda_{k}+b)^2}\right)\times
\]
\[
\times\left(\frac{y'_{k + 1}(1)}{a\lambda_{k}+b}-\frac{ay'_{k}(1)}{(a\lambda_{k}+b)^2}\right),
\]
which proves (5.4).
\end{proof}

Note that the associated function $y_{k+1}^*$ also satisfies Lemma 3.1 and Lemma 3.2.
We will now focus on the case of a triple eigenvalue. It should be noted that if $\lambda_k\neq-\frac{d}{c}$ is a triple eigenvalue, then $T_k = 0$ and $Q_k \neq 0$; consequently, $y_{k+1}^{*}$ does not exist. Similarly, if $\lambda_k=-\frac{d}{c}$ is a triple eigenvalue, then $S_k = 0$ and $P_k \neq 0$; consequently, $y_{k+1}^{*}$ again does not exist.

\begin{lem}
Suppose that $\lambda_k\neq-\frac{d}{c}$ is an eigenvalue of multiplicity three. Then there exists an associated function
$y^{\#}_{k+1}=y_{k+1}+C_{2}y_{k}$, with $C_2$ being a constant, such that
\[
(y_{k+1}^{\#},y_{k+2})=-(ad-bc)\left(\frac{y^{\#}_{k + 1}(1)}{c\lambda_{k}+d}-\frac{cy_{k}(1)}{(c\lambda_{k}+d)^2}\right)\times
\]
\[
\times\left(\frac{y_{k + 2}(1)}{c\lambda_{k}+d}-\frac{cy_{k+1}(1)}{(c\lambda_{k}+d)^2}+\frac{c^2y_{k}(1)}{(c\lambda_{k}+d)^3}\right). \eqno(5.7)
\]

\end{lem}

\begin{proof}
We add
\[
(y_{k+1},y_{k+2})=-(ad-bc)\left( \frac{y_{k + 1}(1)}{c\lambda_{k}+d} - \frac{cy_{k}(1)}{{(c\lambda_{k}+d)}^{2}}\right)\times
\]
\[\times\left(\frac{y_{k + 2}(1)}{c\lambda_{k} + d} - \frac{cy_{k+1}(1)}{\left( c\lambda_{k} + d \right)^{2}}
+\frac{c^2y_{k}(1)}{\left( c\lambda_{k} + d \right)^{3}}\right)+L_{k},\eqno(5.8)
\] 
where $L_k$ is defined  by (5.8), to (4.9), multiplied by $C_2=-\frac{L_k}{Q_k}$, and obtain
\[
(y_{k + 1}+C_2y_{k},y_{k+2})=-(ad-bc)\left(\frac{y_{k + 1}(1)+C_2y_{k}(1)}{c\lambda_{k}+d}-\frac{cy_{k}(1)}{(c\lambda_{k}+d)^2}\right)\times
\]
\[
\times\left(\frac{y_{k + 2}(1)}{c\lambda_{k}+d}-\frac{cy_{k+1}(1)}{(c\lambda_{k}+d)^2}+\frac{c^2y_{k}(1)}{(c\lambda_{k}+d)^3}\right),\]
which completes the proof.
\end{proof}
\begin{lem}
Suppose that $\lambda_k=-\frac{d}{c}$ is an eigenvalue of multiplicity three.  Then there exists an associated function
$y^{\#}_{k+1}=y_{k+1}+C_{2}y_{k}$, with $C_2$ being a constant, such that
\[
(y_{k+1}^{\#},y_{k+2})=-(ad-bc)\left(\frac{(y^{\#}_{k + 1})'(1)}{a\lambda_{k}+b}-\frac{ay'_{k}(1)}{(a\lambda_{k}+b)^2}\right)\times
\]
\[
\times\left(\frac{y'_{k + 2}(1)}{a\lambda_{k}+b}-\frac{ay'_{k+1}(1)}{(a\lambda_{k}+b)^2}+\frac{a^2y'_{k}(1)}{(a\lambda_{k}+b)^3}\right). \eqno(5.9)
\],

\end{lem}

\begin{proof}
We add
\[
(y_{k+1},y_{k+2})=-(ad-bc)\left( \frac{y'_{k + 1}(1)}{a\lambda_{k} + b} - \frac{ay'_{k}(1)}{{(a\lambda_{k} + b)}^{2}}\right)\times
\]
\[\times\left(\frac{y'_{k + 2}(1)}{a\lambda_{k} + b} - \frac{ay'_{k+1}(1)}{\left( a\lambda_{k} + b \right)^{2}}
+\frac{a^2y'_{k}(1)}{\left( a\lambda_{k} + b\right)^{3}}\right)+M_{k}, \eqno(5.10)
\] 
where $M_k$ is defined  by (5.10), to (4.11), multiplied by $C_2=-\frac{M_k}{P_k}$, and obtain (5.9).
\end{proof}
Note that if $\lambda_{k}\neq-\frac{d}{c}$, then $y^{\#}_{k+1}$ also satisfies Lemma 3.1, Lemma 4.2, and 
\[
(y^{\#}_{k+1}, y_{k+1}) =-(ad-bc)\left(\frac{y^{\#}_{k + 1}(1)}{c\lambda_{k}+d}-\frac{cy_{k}(1)}{(c\lambda_{k}+d)^2}\right)\times
\]
\[
\times\left(\frac{y_{k + 1}(1)}{c\lambda_{k}+d}-\frac{cy_{k}(1)}{(c\lambda_{k}+d)^2}\right)+Q_{k}. \eqno(5.11)
\]
Similarly, if $\lambda_{k}=-\frac{d}{c}$, then 
\[
(y^{\#}_{k+1}, y_{k+1}) =-(ad-bc)\left(\frac{(y^{\#}_{k + 1})'(1)}{a\lambda_{k}+b}-\frac{ay'_{k}(1)}{(a\lambda_{k}+b)^2}\right)\times
\]
\[
\times\left(\frac{y'_{k + 1}(1)}{a\lambda_{k}+b}-\frac{ay'_{k}(1)}{(a\lambda_{k}+b)^2}\right)+P_{k}. \eqno(5.12)
\] 

\section{Existence of special associated functions $y_{k+2}^*$ and $y_{k+2}^\#$.} 
Observe that the function $y^{*}_{k+2}$, given by
\(y^{*}_{k+2}=y_{k+2}+C_{2}y_{k+1}\),
with the same constant $C_{2}$, satisfies Lemma 4.1 and Lemma 4.2 in the following form.

\noindent(a) If \(\lambda_{k} \neq - \frac{d}{c}\neq\lambda_n\), then
$$
\left( y^{*}_{k + 2},y_{n} \right) = - (ad - bc)\left(\frac{y^{*}_{k + 2}(1)}{c\lambda_{k} + d} - \frac{cy^{\#}_{k+1}(1)}{\left( c\lambda_{k} + d \right)^{2}}
+\frac{c^2y_{k}(1)}{\left( c\lambda_{k} + d \right)^{3}}\right)\cdot\frac{y_{n}(1)}{ c\lambda_{n} + d}.\eqno(6.1)
$$
(b) If \(\lambda_{k}\neq - \frac{d}{c}=\lambda_n\), then
\[\left( y^{*}_{k + 2},y_{n} \right) = - (ad - bc)\left(\frac{y^{*}_{k + 2}(1)}{c\lambda_{k} + d} -\frac{cy^{\#}_{k+1}(1)}{\left( c\lambda_{k} + d \right)^{2}}
+\frac{c^2y_{k}(1)}{\left( c\lambda_{k} + d \right)^{3}}\right)
\cdot \frac{y'_{n}(1)}{ a\lambda_{n} + b}. \eqno(6.2)\]
(c) If \(\lambda_{k} = - \frac{d}{c}\neq\lambda_n\), then
\[\left( y^{*}_{k+2},y_{n} \right) =- (ad - bc)\left(\frac{(y^{*}_{k + 2})'(1)}{a\lambda_{k} + b} - \frac{a(y^{\#}_{k+1})'(1)}{\left( a\lambda_{k} + b\right)^{2}}
+\frac{a^2y'_{k}(1)}{\left( a\lambda_{k} + b \right)^{3}}\right)\cdot\frac{y_{n}(1)}{ c\lambda_{n} + d}\eqno(6.3)\]
Similarly, the function $y^{*}_{k+2}$ satisfies Lemma 4.3 in the following form.

\noindent(a) If $\lambda _{k}\ne -\frac{d}{c}$, then
$$
(y^{*}_{k+2},y_{k})=- (ad - bc)\left(\frac{y^{*}_{k + 2}(1)}{c\lambda_{k} + d} - \frac{cy^{\#}_{k+1}(1)}{\left( c\lambda_{k} + d \right)^{2}}
+\frac{c^2y_{k}(1)}{\left( c\lambda_{k} + d \right)^{3}}\right)\cdot\frac{y_{k}(1)}{ c\lambda_{k} + d}+Q_{k}. \eqno(6.4)
$$
(b) If $\lambda _{k}=-\frac{d}{c}$, then 
$$(y^{*}_{k+2},y_{k})=- (ad - bc)\left(\frac{(y^{*}_{k + 2})'(1)}{a\lambda_{k} + b} - \frac{a(y^{\#}_{k+1})'(1)}{\left( a\lambda_{k} + b\right)^{2}}
+\frac{a^2y'_{k}(1)}{\left( a\lambda_{k} + b \right)^{3}}\right)\cdot\frac{y'_{k}(1)}{ a\lambda_{k} + b}+P_{k}.\eqno(6.5)
$$
For the cases $\lambda _{k}\ne -\frac{d}{c}$ and $\lambda _{k}=-\frac{d}{c}$ multiplying (5.2) and (5.5) by $C_2$ and adding to (5.8) and (5.10), respectively, gives
\[
(y^{*}_{k+2},y_{k+1})=-(ad-bc)\left(\frac{y^{*}_{k + 2}(1)}{c\lambda_{k} + d} - \frac{cy^{\#}_{k+1}(1)}{\left( c\lambda_{k} + d \right)^{2}}
+\frac{c^2y_{k}(1)}{\left( c\lambda_{k} + d \right)^{3}}\right)\times
\]
\[\times\left( \frac{y_{k + 1}(1)}{c\lambda_{k}+d} - \frac{cy_{k}(1)}{{(c\lambda_{k}+d)}^{2}}\right),\eqno(6.6)
\] 
\[
(y^{*}_{k+2},y_{k+1})=-(ad-bc)\left(\frac{(y^{*}_{k + 2})'(1)}{a\lambda_{k} + b} - \frac{a(y^{\#}_{k+1})'(1)}{\left( a\lambda_{k} + b \right)^{2}}
+\frac{a^2y'_{k}(1)}{\left( a\lambda_{k} + b \right)^{3}}\right)\times
\]
\[\times\left( \frac{y'_{k + 1}(1)}{a\lambda_{k}+b} - \frac{ay'_{k}(1)}{{(a\lambda_{k}+b)}^{2}}\right).\eqno(6.7)
\] 
\begin{lem}
Suppose that $\lambda_k\neq-\frac{d}{c}$ is an eigenvalue of multiplicity three.  Then there exists an associated function
$y^{\#}_{k+2}=y^{*}_{k+2}+D_{1}y_{k}$, with $D_1$ being a constant, such that
\[
(y_{k+2}^{\#},y_{k+2})=-(ad-bc)\left(\frac{y^{\#}_{k + 2}(1)}{c\lambda_{k}+d}-\frac{cy^{\#}_{k+1}(1)}{(c\lambda_{k}+d)^2}+\frac{c^2y_{k}(1)}{(c\lambda_{k}+d)^3}\right)\times
\]
\[
\times\left(\frac{y_{k + 2}(1)}{c\lambda_{k}+d}-\frac{cy_{k+1}(1)}{(c\lambda_{k}+d)^2}+\frac{c^2y_{k}(1)}{(c\lambda_{k}+d)^3}\right). \eqno(6.8)
\]

\end{lem}

\begin{proof}
We add
\[
(y^{*}_{k+2},y_{k+2})=-(ad-bc)\left(\frac{y^{*}_{k + 2}(1)}{c\lambda_{k}+d}-\frac{cy^{\#}_{k+1}(1)}{(c\lambda_{k}+d)^2}+\frac{c^2y_{k}(1)}{(c\lambda_{k}+d)^3}\right)\times
\]
\[
\times\left(\frac{y_{k + 2}(1)}{c\lambda_{k}+d}-\frac{cy_{k+1}(1)}{(c\lambda_{k}+d)^2}+\frac{c^2y_{k}(1)}{(c\lambda_{k}+d)^3}\right)+J_{k},\eqno(6.9)
\] 
where $J_k$ is defined  by (6.9), to (4.9), multiplied by $D_1=-\frac{J_k}{Q_k}$, and obtain (6.8).
\end{proof}
\begin{lem}
Suppose that $\lambda_k=-\frac{d}{c}$ is an eigenvalue of multiplicity three.  Then there exists an associated function
$y^{\#}_{k+2}=y^{*}_{k+2}+D_{1}y_{k}$, with $D_1$ being a constant, such that
\[
(y_{k+2}^{\#},y_{k+2})=-(ad-bc)\left(\frac{(y^{\#}_{k + 2})'(1)}{a\lambda_{k}+b}-\frac{a(y^{\#}_{k+1})'(1)}{(a\lambda_{k}+b)^2}+\frac{a^2y'_{k}(1)}{(a\lambda_{k}+b)^3}\right)\times
\]
\[
\times\left(\frac{y'_{k + 2}(1)}{a\lambda_{k}+b}-\frac{ay'_{k+1}(1)}{(a\lambda_{k}+b)^2}+\frac{a^2y'_{k}(1)}{(a\lambda_{k}+b)^3}\right). \eqno(6.10)
\]

\end{lem}

\begin{proof}
We add
\[
(y^{*}_{k+2},y_{k+2})=-(ad-bc)\left(\frac{(y^{*}_{k + 2})'(1)}{a\lambda_{k} + b} - \frac{a(y^{\#}_{k+1})'(1)}{\left( a\lambda_{k} + b \right)^{2}}
+\frac{a^2y'_{k}(1)}{\left( a\lambda_{k} + b\right)^{3}}\right)\times
\]
\[\times\left(\frac{y'_{k + 2}(1)}{a\lambda_{k} + b} - \frac{ay'_{k+1}(1)}{\left( a\lambda_{k} + b \right)^{2}}
+\frac{a^2y'_{k}(1)}{\left( a\lambda_{k} + b\right)^{3}}\right)+H_{k}, \eqno(6.11)
\] 
where $H_k$ is defined  by (6.11), to (4.11), multiplied by $D_1=-\frac{H_k}{P_k}$, and obtain (6.10).
\end{proof}
Note that for the function $y^{\#}_{k+2}$ the equalities (6.1)-(6.7) are also true and, instead of writing them again, we will refer to these equalities when we will need them for $y^{\#}_{k+2}$.

\section{Basis properties of the root functions in the cases (i) and (ii).}
For the sake of simplicity  let us introduce the following notations:

\[
\mathfrak{A}(y_n) =
\begin{cases}
\displaystyle  
\dfrac{y_n(1)}{c\lambda_n + d}, 
& \text{if }  \lambda_n \neq -\dfrac{d}{c}, \\[10pt]
\displaystyle \ 
\dfrac{y_n'(1)}{a\lambda_n + b}, 
& \text{if }  \lambda_n = -\dfrac{d}{c}.
\end{cases}\eqno(7.1)
\] 
Similarly, 
\[
\mathfrak{A}(y_{k+1}) =
\begin{cases}
\displaystyle  
\frac{y_{k + 1}(1)}{c\lambda_{k}+d} - \frac{cy_{k}(1)}{{(c\lambda_{k}+d)}^{2}}, 
& \text{if }  \lambda_k \neq -\dfrac{d}{c}, \\[10pt]
\displaystyle \ 
\frac{y'_{k + 1}(1)}{a\lambda_{k}+b} - \frac{ay'_{k}(1)}{{(a\lambda_{k}+b)}^{2}}, 
& \text{if }  \lambda_k = -\dfrac{d}{c},
\end{cases}\eqno(7.2)
\] 
\[
\mathfrak{A}(y_{k+2}) =
\begin{cases}
\displaystyle  
\frac{y_{k + 2}(1)}{c\lambda_{k} + d} - \frac{cy_{k+1}(1)}{\left( c\lambda_{k} + d \right)^{2}}
+\frac{c^2y_{k}(1)}{\left( c\lambda_{k} + d \right)^{3}}, 
& \text{if }  \lambda_k \neq -\dfrac{d}{c}, \\[10pt]
\displaystyle \ 
\frac{y'_{k + 2}(1)}{a\lambda_{k} + b} - \frac{ay'_{k+1}(1)}{\left( a\lambda_{k} + b \right)^{2}}
+\frac{a^2y'_{k}(1)}{\left( a\lambda_{k} + b\right)^{3}}, 
& \text{if }  \lambda_k = -\dfrac{d}{c}.
\end{cases}\eqno(7.3)
\] 
In particular,
\[
\mathfrak{A}(y^{*}_{k+1}) =
\begin{cases}
\displaystyle  
\frac{y^{*}_{k + 1}(1)}{c\lambda_{k}+d} - \frac{cy_{k}(1)}{{(c\lambda_{k}+d)}^{2}}, 
& \text{if }  \lambda_k \neq -\dfrac{d}{c}, \\[10pt]
\displaystyle \ 
\frac{(y^{*}_{k + 1})'(1)}{a\lambda_{k}+b} - \frac{ay'_{k}(1)}{{(a\lambda_{k}+b)}^{2}}, 
& \text{if }  \lambda_k = -\dfrac{d}{c},
\end{cases}\eqno(7.4)
\] 
\[
\mathfrak{A}(y^{\#}_{k+1}) =
\begin{cases}
\displaystyle  
\frac{y^{\#}_{k + 1}(1)}{c\lambda_{k}+d} - \frac{cy_{k}(1)}{{(c\lambda_{k}+d)}^{2}}, 
& \text{if }  \lambda_k \neq -\dfrac{d}{c}, \\[10pt]
\displaystyle \ 
\frac{(y^{\#}_{k + 1})'(1)}{a\lambda_{k}+b} - \frac{ay'_{k}(1)}{{(a\lambda_{k}+b)}^{2}}, 
& \text{if }  \lambda_k = -\dfrac{d}{c},
\end{cases}\eqno(7.5)
\] 
\[
\mathfrak{A}(y^{*}_{k+2}) =
\begin{cases}
\displaystyle  
\frac{y^{*}_{k + 2}(1)}{c\lambda_{k} + d} - \frac{cy^{\#}_{k+1}(1)}{\left( c\lambda_{k} + d \right)^{2}}
+\frac{c^2y_{k}(1)}{\left( c\lambda_{k} + d \right)^{3}}, 
& \text{if }  \lambda_k \neq -\dfrac{d}{c}, \\[10pt]
\displaystyle \ 
\frac{(y^{*}_{k + 2})'(1)}{a\lambda_{k} + b} - \frac{a(y^{\#}_{k+1})'(1)}{\left( a\lambda_{k} + b \right)^{2}}
+\frac{a^2y'_{k}(1)}{\left( a\lambda_{k} + b\right)^{3}}, 
& \text{if }  \lambda_k = -\dfrac{d}{c},
\end{cases}\eqno(7.6)
\] 
\[
\mathfrak{A}(y^{\#}_{k+2}) =
\begin{cases}
\displaystyle  
\frac{y^{\#}_{k + 2}(1)}{c\lambda_{k} + d} - \frac{cy^{\#}_{k+1}(1)}{\left( c\lambda_{k} + d \right)^{2}}
+\frac{c^2y_{k}(1)}{\left( c\lambda_{k} + d \right)^{3}}, 
& \text{if }  \lambda_k \neq -\dfrac{d}{c}, \\[10pt]
\displaystyle \ 
\frac{(y^{\#}_{k + 2})'(1)}{a\lambda_{k} + b} - \frac{a(y^{\#}_{k+1})'(1)}{\left( a\lambda_{k} + b \right)^{2}}
+\frac{a^2y'_{k}(1)}{\left( a\lambda_{k} + b\right)^{3}}, 
& \text{if }  \lambda_k = -\dfrac{d}{c}.
\end{cases}\eqno(7.7)
\] 

\subsection{Case (i).}

\begin{thm} Assume that all eigenvalues of problem (1.1)--(1.3) are real and simple. Then the system
\[
\{ y_n \} \quad (n = 0,1,\ldots;\; n \neq l), \tag{7.8}
\]
where $l$ is a nonnegative integer, is a basis of the space $L_p(0,1)$ $(1<p<+\infty)$.

\end{thm}

\begin{proof}
It is sufficient to establish the existence of a system
\[
\{ u_n \} \quad (n = 0,1,\ldots;\; n \neq l), \tag{7.9}
\]
which is biorthogonal to the system \emph{(7.8)}. Since $B_n \neq 0$, we introduce the elements of the system \emph{(7.9)} by
\[
u_n(x)=\frac{1}{B_n \mathfrak{A}(y_l)}
\begin{vmatrix}
y_n(x) & \mathfrak{A}(y_n)\\
y_l(x) & \mathfrak{A}(y_l)
\end{vmatrix}. \tag{7.10}
\]
Using relations \emph{(2.4)}, \emph{(2.5)}, and \emph{(2.16)}, it remains to verify that
\[
(u_n, y_m) = \delta_{nm}, \tag{7.11}
\]
where $\delta_{nm}$ $(n,m=0,1,\ldots;\; n,m \neq l)$ denotes the Kronecker delta, that is,
$\delta_{nm}=0$ for $n \neq m$ and $\delta_{nn}=1$. Indeed, if $n\neq m$, then
\[
(u_n, y_m) =\left(
\frac{1}{B_n \mathfrak{A}(y_l)}
\begin{vmatrix}
y_n(x) & \mathfrak{A}(y_n)\\
y_l(x) & \mathfrak{A}(y_l) 
\end{vmatrix}, y_m\right)=
\frac{1}{B_n \mathfrak{A}(y_l)}
\begin{vmatrix}
(y_n,y_m) & \mathfrak{A}(y_n)\\
(y_l,y_m)& \mathfrak{A}(y_l) 
\end{vmatrix}
\]
\[
\frac{1}{B_n \mathfrak{A}(y_l)}
\begin{vmatrix}
-(ad-bc)\mathfrak{A}(y_n) {{\mathfrak{A}(y_m)}} & \mathfrak{A}(y_n)\\
-(ad-bc)\mathfrak{A}(y_l) {{\mathfrak{A}(y_m)}} & \mathfrak{A}(y_l) 
\end{vmatrix}=0.
\]
Similarly, 
\[
(u_n, y_n) =\left(
\frac{1}{B_n \mathfrak{A}(y_l)}
\begin{vmatrix}
y_n(x) & \mathfrak{A}(y_n)\\
y_l(x) & \mathfrak{A}(y_l) 
\end{vmatrix}, y_m\right)=
\frac{1}{B_n \mathfrak{A}(y_l)}
\begin{vmatrix}
(y_n,y_n) & \mathfrak{A}(y_n)\\
(y_l,y_n)& \mathfrak{A}(y_l) 
\end{vmatrix}
\]
\[
=\frac{1}{B_n \mathfrak{A}(y_l)}
\begin{vmatrix}
 -(ad-bc)\mathfrak{A}(y_n) {{\mathfrak{A}(y_n)}}+B_n& \mathfrak{A}(y_n)\\
-(ad-bc)\mathfrak{A}(y_l) {{\mathfrak{A}(y_n)}}& \mathfrak{A}(y_l) 
\end{vmatrix}=1.
\]
The remaining part of the proof is done in analogy with \cite{kerimov} using the asymptotic formula (1.4) and Bari's theorem \cite{bary} (see e.g. \cite{kost}).
\end{proof}

\subsection{Case (ii).}
\begin{thm}
If $\lambda_k$ is an eigenvalue of multiplicity two, then the system
\[
\{y_n\} \quad (n = 0,1,\ldots;\ n \neq k+1), \tag{7.12}
\]
is a basis of the space $L_p(0,1)$ $(1<p<+\infty)$.

\end{thm}

\begin{proof} The biorthogonal system in this case is defined by
\[
u_n(x)=\frac{1}{B_n \mathfrak{A}(y_k)}
\begin{vmatrix}
y_n(x) & \mathfrak{A}(y_n)\\
y_k(x) & \mathfrak{A}(y_k)
\end{vmatrix}\ (n \ne k,\ k+1), \tag{7.13}
\]
\[
u_k(x)=\frac{1}{\beta_k\mathfrak{A}(y_k)}
\begin{vmatrix}
y_{k+1}(x) & \mathfrak{A}(y_{k+1})\\
y_k(x) & \mathfrak{A}(y_k)
\end{vmatrix},\eqno(7.14)
\]
where $\beta_k=T_k$ if $\lambda_k\neq-\frac{d}{c}$, and $\beta_k=S_k$ if $\lambda_k=-\frac{d}{c}$. The proof of (7.11) for $n,m\neq k+1$ is done using 
Lemma 2.1, formulas (2.13), (2.14), (2.16), (3.6), and (3.7).
\end{proof}

\begin{thm} If $\lambda_k$ is an eigenvalue of multiplicity two, then the system
$$
\left\{ y_{n}\right\}  \ (n=0,\ 1,\ldots ; n \ne k), \eqno(7.15)
$$
is a basis of the space $L_p(0,1)$ $(1<p<+\infty)$ if and only if  $\mathfrak{A}(y_{k+1}^{*})\ne 0$.
\end{thm}

\begin{proof} The biorthogonal system is
defined as 
$$
u_{n}(x)=\frac{1}{B_{n} \mathfrak{A}(y_{k+1}^{*})}\begin{vmatrix} y_{n}(x) & \mathfrak{A}(y_{n}) \\ y_{k+1}^{*}(x) & \mathfrak{A}(y_{k+1}^{*}) \end{vmatrix}\ (n \ne k,\ k+1), \eqno(7.16)
$$
$$
u_{k+1}(x)=\frac{1}{\beta_k\mathfrak{A}(y_{k+1}^{*})}\begin{vmatrix} y_{k}(x) & \mathfrak{A}(y_{k})\\ y_{k+1}^{*}(x) & \mathfrak{A}(y_{k+1}^{*}) \end{vmatrix}. \eqno(7.17)
$$
The proof of (7.11) for $n,m\neq k$ is done using 
Lemmas 2.1, 3.1, 5.1, 5.2, formulas (2.16), (3.6), (3.7), (5.1) and (5.4).

If $\mathfrak{A}(y_{k+1}^{*}) = 0$, then, in view of relations 
(3.1)--(3.3), (5.1), and (5.4), the function $y_{k+1}^{*}$ is orthogonal to 
all elements of the system (7.15). Consequently, the system (7.15) fails 
to be complete in the space $L_{2}(0,1)$.
\end{proof}
\begin{thm}If $\lambda_k$ is an eigenvalue of multiplicity two, then the system
$$
\left\{ y_{n}\right\}  \ (n=0,\ 1,\ldots ; n \ne l), \eqno(7.18)
$$
where $l \ne k,\ k+1$ is a non-negative integer, is a basis of the space $L_p(0,1)$ $(1<p<+\infty)$.
\end{thm}

\begin{proof} The biorthogonal system coincides with the expression given in (7.10) for 
$n \neq k, k+1$, and
$$
u_{k}(x)=\frac{1}{\beta_k \mathfrak{A}(y_{l})}\begin{vmatrix} y_{k+1}^{*}(x) & \mathfrak{A}(y_{k+1}^{*}) \\ y_{l}(x) & \mathfrak{A}(y_{l}) \end{vmatrix}, \eqno(7.19)
$$
$$
u_{k+1}(x)=\frac{1}{\beta_k \mathfrak{A}(y_{l})}\begin{vmatrix} y_{k}(x) & \mathfrak{A}(y_{k}) \\ y_{l}(x) & \mathfrak{A}(y_{l}) \end{vmatrix}.\eqno(7.20)
$$

\end{proof}
\section{Basis properties of the root functions in the cases (iii) and (iv).}
\subsection{Case (iii).}

\begin{thm}If $\lambda_k$ is an eigenvalue of multiplicity three, then the system
$$
\left\{ y_{n}\right\}  \ (n=0,\ 1,\ldots ; n \ne k+2), \eqno(8.1)
$$
is a basis of the space $L_p(0,1)$ $(1<p<+\infty)$
\end{thm}

\begin{proof} The biorthogonal system is given by the formula
(7.13) for $n \ne k,\ k+1,\ k+2$, and
$$
u_{k}(x)=\frac{1}{\gamma_k \mathfrak{A}(y_{k})}\begin{vmatrix} y_{k+2}^{\#}(x) & \mathfrak{A}(y_{k+2}^{\#}) \\ y_{k}(x) & \mathfrak{A}(y_{k})\end{vmatrix}, \eqno(8.2)
$$
$$
u_{k+1}(x)=\frac{1}{\gamma_k \mathfrak{A}(y_{k})}\begin{vmatrix} y_{k+1}(x) & \mathfrak{A}(y_{k+1}) \\ y_{k}(x) & \mathfrak{A}(y_{k})\end{vmatrix}. \eqno(8.3)
$$
where $\gamma_k=Q_k$ if $\lambda_k\neq-\frac{d}{c}$, and $\gamma_k=P_k$ if $\lambda_k=-\frac{d}{c}$. The proof of (7.11) for $n,m\neq k+2$ is done using 
Lemmas 4.1-4.3, formulas (5.2), (5.5), (6.6) and (6.7).
\end{proof}

\begin{thm} 
If $\lambda_k$ is an eigenvalue of multiplicity three,  then the system
$$
\left\{ y_{n}\right\}  \ (n=0,\ 1,\ldots ; n \ne k+1), \eqno(8.4)
$$
is a basis of the space $L_p(0,1)$ $(1<p<+\infty)$ if and only if $\mathfrak{A}(y_{k+1}^{\#})\ne 0$.
\end{thm}

\begin{proof}  The biorthogonal 
system is defined by
$$
u_{n}(x)=\frac{1}{B_{n} \mathfrak{A}(y_{k+1}^{\#})}\begin{vmatrix} y_{n}(x) & \mathfrak{A}(y_{n}) \\ y_{k+1}^{\#}(x) & \mathfrak{A}(y_{k+1}^{\#})\end{vmatrix}\ (n \ne k,\ k+1,\ k+2), \eqno(8.5)
$$
$$
u_{k}(x)=\frac{1}{\gamma_k \mathfrak{A}(y_{k+1}^{\#})}\begin{vmatrix}y_{k+2}^{\#}(x) & \mathfrak{A}(y_{k+2}^{\#}) \\ y_{k+1}^{\#}(x) & \mathfrak{A}(y_{k+1}^{\#}) \end{vmatrix}, \eqno(8.6)
$$
$$
u_{k+2}(x)=\frac{1}{\gamma_k \mathfrak{A}(y_{k+1}^{\#})}\begin{vmatrix} y_{k}(x) & \mathfrak{A}(y_{k}) \\ y_{k+1}^{\#}(x) & \mathfrak{A}(y_{k+1}^{\#}) \end{vmatrix}. \eqno(8.7)
$$
The proof of (7.11) for $n,m\neq k+1$ is done using formulas (2.13), (2.14), (2.16), (3.1)-(3.3), (4.7), (4.8), (4.9), (4.11), (5.7), (5.9), (6.4), (6.5), (6.8), and (6.10).

By an argument analogous to that of  Theorem~7.3, one can show that if
\(
\mathfrak{A}\bigl(y_{k+1}^{\#}\bigr)=0,
\)
then the function $y_{k+1}^{\#}(x)$ is orthogonal to all elements of the
system~(8.4). Therefore, the system~(8.4) is not complete and consequently it is not a basis in $L_p(0,1)$ $(1<p<+\infty)$.
\end{proof}
\begin{thm} If $\lambda_k$ is an eigenvalue of multiplicity three,   then the system
$$
\left\{ y_{n}\right\}  \ (n=0,\ 1,\ldots ; n \ne k), \eqno(8.8)
$$
is a basis of the space $L_p(0,1)$ $(1<p<+\infty)$ if and only if $\mathfrak{A}(y_{k+2}^{\#})\ne 0$.
\end{thm}

\begin{proof} We define the elements of the biorthogonal system
by
$$
u_{n}(x)=\frac{1}{B_{n} \mathfrak{A}(y_{k+2}^{\#})}\begin{vmatrix} y_{n}(x) & \mathfrak{A}(y_{n}) \\ y_{k+2}^{\#}(x) & \mathfrak{A}(y_{k+2}^{\#}) \end{vmatrix}\ (n \ne k,\ k+1,\ k+2),
\eqno(8.9)
$$
$$
u_{k+1}(x)=\frac{1}{\gamma_k\mathfrak{A}(y_{k+2}^{\#})}\begin{vmatrix} y_{k+1}(x) & \mathfrak{A}(y_{k+1}) \\ y_{k+2}^{\#}(x) & \mathfrak{A}(y_{k+2}^{\#}) \end{vmatrix}, \eqno(8.10)
$$
$$
u_{k+2}(x)=\frac{1}{\gamma_k  \mathfrak{A}(y_{k+2}^{\#})}\begin{vmatrix} y_{k}(x) & \mathfrak{A}(y_{k}) \\ y_{k+2}^{\#}(x) & \mathfrak{A}(y_{k+2}^{\#}) \end{vmatrix}. \eqno(8.11)
$$
If $\mathfrak{A}(y_{k+2}^{\#})=0$, then the system (8.8) fails to be complete.
\end{proof}

\begin{thm}
 If $\lambda_k$ is an eigenvalue of multiplicity three,  then the system
$$
\left\{ y_{n}\right\}  \ (n=0,\ 1,\ldots ; n \ne l), \eqno(8.12)
$$
where $l \ne k,\ k+1,\ k+2$ is a non-negative integer, is a basis of the space $L_p(0,1)$ $(1<p<+\infty)$.
\end{thm}

\begin{proof} The biorthogonal system is defined by the formula (7.10) for $n \ne k,\ k+1,\ k+2,\ l$, and
$$
u_{k}(x)=\frac{1}{\gamma_k \mathfrak{A}y_{l}}\begin{vmatrix} y_{k+2}^{\#}(x) & \mathfrak{A}(y_{k+2}^{\#}) \\ y_{l}(x) & \mathfrak{A}(y_{l}) \end{vmatrix}. \eqno(8.13)
$$
$$
u_{k+1}(x)=\frac{1}{\gamma_k \mathfrak{A}(y_{l})}\begin{vmatrix} y_{k+1}^{\#}(x) & \mathfrak{A}(y_{k+1}^{\#}) \\ y_{l}(x) & \mathfrak{A}(y_{l}) \end{vmatrix}, \eqno(8.14)
$$
$$
u_{k+2}(x)=\frac{1}{\gamma_k \mathfrak{A}(y_{l})}\begin{vmatrix} y_{k}(x) & \mathfrak{A}(y_{k}) \\ y_{l}(x) & \mathfrak{A}(y_{l}) \end{vmatrix}. \eqno(8.15)
$$
\end{proof}

\subsection{Case (iv).}

\begin{thm} Let $\lambda_r$ and $\lambda_s=\overline{\lambda_r}$ be a pair of
complex conjugate non-real eigenvalues. Then each of the systems
\[
\left\{ y_n \right\} \quad (n=0,1,\ldots;\ n\neq r), \eqno(8.16)
\]
and
\[
\left\{ y_n \right\} \quad (n=0,1,\ldots;\ n\neq l), \eqno(8.17)
\]
where $l$ is a non-negative integer with $l\neq r,s$, is a basis of the space $L_p(0,1)$ $(1<p<+\infty)$.

\end{thm}

\begin{proof} For (8.16) the biorthogonal system is 
$$
u_{n}(x)=\frac{1}{B_{n} \mathfrak{A}(y_{s})}\begin{vmatrix} y_{n}(x) & \mathfrak{A}(y_{n}) \\ y_{s}(x) & \mathfrak{A}(y_{s})\end{vmatrix}\ (n \ne r,\ s), \eqno(8.18)
$$
$$
u_{s}(x)=\frac{1}{T_r\mathfrak{A}(y_{s}) }\begin{vmatrix} y_{r}(x) & \mathfrak{A}(y_{r}) \\ y_{s}(x) & \mathfrak{A}(y_{s}) \end{vmatrix}. \eqno(8.19)
$$

The biorthogonal system of (8.17) is defined by (7.10) for $n \ne
r,s,l$ and
$$
u_{r}(x)=\frac{1}{T_s\mathfrak{A}(y_{l}) }\begin{vmatrix} y_{s}(x) & \mathfrak{A}(y_{s}) \\ y_{l}(x) & \mathfrak{A}(y_{l}) \end{vmatrix}. \eqno(8.20)
$$
$$
u_{s}(x)=\frac{1}{T_r\mathfrak{A}(y_{l}) }\begin{vmatrix} y_{r}(x) & \mathfrak{A}(y_{r}) \\ y_{l}(x) & \mathfrak{A}(y_{l}) \end{vmatrix}. \eqno(8.21)
$$
\end{proof}

\section{Examples.} Our first example corresponds to the double eigenvalue with $\lambda _{k}\neq-\frac{d}{c}$.

\subsection*{Example 1.}Let us take the problem  
$$
-y^{\prime \prime }=\lambda y,\ 0<x<1,
$$
$$
y'(0)=0,\ 
{\lambda}y(1)=\left(\frac{4}{\pi^2}{\lambda }-1\right)y^{\prime }(1).
$$
Note that $a=1$, $b=0$, $c=\frac{4}{\pi^2}$, $d=-1$, $ad-bc=-1<0$, $q(x)\equiv0$. We obtain that
$\lambda _{0}=\lambda _{1}=0$ is the eigenvalue of multiplicity 2, and $\lambda _{0}\neq-\frac{d}{c}$. The next eigenvalue $\lambda _{2}=\frac{\pi^2}{4}$ is simple and $\lambda _{2}=-\frac{d}{c}$. The remaining eigenvalues are $\frac{\pi^2}{4}<\lambda _{3}<\lambda _{4}<\ldots$. The eigenfunctions are $y_{0}=1$, $y_{2}=\cos{\frac{\pi x}{2}}$, $y_{n}=\cos{\sqrt{\lambda _{n}}x}$ $(n\ge 3)$, and the first associated function of $y_{0}$ satisfies
$$
-y_1^{\prime \prime }=\lambda y_1+y_0,\ 0<x<1,
$$
$$
y_1'(0)=0,\ 
{\lambda}y_1(1)+y_0(1)=\left(\frac{4}{\pi^2}{\lambda }-1\right)y_1^{\prime }(1)+\frac{4}{\pi^2}y_0^\prime(1).
$$
By solving this boundary value problem we obtain $y_{1}=-\frac{1}{2}x^2+C$, where $C$ is
an arbitrary constant.
By formula (3.6),
$$
T_0:=(y_{1},y_{0})-\left( \frac{y_{1}(1)}{-1} - \frac{4}{\pi^2}\cdot\frac{y_{0}(1)}{{(-1)}^{2}}\right)\cdot \frac{y_{0}(1)}{-1}=\frac{1}{3}-\frac{4}{{\pi}^{2}}.
$$
By (4.10),
$$Q_{0}=(y_{1},y_{1})-\left(\frac{y_{1}(1)}{-1} - \frac{4}{\pi^2}\cdot\frac{y_{0}(1)}{(-1)^{2}}
\right)^2=\frac{1}{20}-\frac{C}{3}+C^{2}-\frac{\left(4+\left(-\frac{1}{2}+C\right) \mathrm{\pi}^{2}\right)^{2}}{\mathrm{\pi}^{4}}.$$
Then $C_{1}=-\frac{Q_{0}}{T_{0}}=\frac{240+\left(-10 C+3\right) \mathrm{\pi}^{4}+\left(120 C-60\right) \mathrm{\pi}^{2}}{5 \mathrm{\pi}^{2} \left(\mathrm{\pi}^{2}-12\right)}$. Therefore, 
$$
y_{1}^{*}=y_{1}+C_{1}y_{0}=\frac{480+\left(-5 x^{2}-10 C+6\right) \mathrm{\pi}^{4}+60 \left(x^{2}+2 C-2\right) \mathrm{\pi}^{2}}{10 \mathrm{\pi}^{2} \left(\mathrm{\pi}^{2}-12\right)}.
$$
By (7.4),
$$
\mathfrak{A}(y^{*}_{k+1}) =\frac{y_{1}^{*}(1)}{-1} -\frac{4}{{\pi}^{2}}\cdot \frac{y_{0}(1)}{(-1)^{2}}=\frac{\left(10 C-1\right) \mathrm{\pi}^{2}-120 C+20}{10 \mathrm{\pi}^{2}-120}.
$$
Consequently, by Theorem 7.3,
the system $\left\{ y_{1},\ y_{2},\ y_{3},\ldots\right\}$,  with removed $y_{0}$ is a basis  in $L_p(0,1)$ $(1<p<+\infty)$ if and only if $C\ne\frac{\mathrm{\pi}^{2}-20}{10 \left(\mathrm{\pi}^{2}-12\right)}$.
If $C= \frac{\mathrm{\pi}^{2}-20}{10 \left(\mathrm{\pi}^{2}-12\right)}$, then $y_{1}^{*}(x)=\frac{-x^{2} \mathrm{\pi}^{2}+\mathrm{\pi}^{2}-8}{2 \mathrm{\pi}^{2}}$ is orthogonal to all the elements of the system $\left\{ y_{n}\right\}  \ (n=1,\ 2,\ 3,\ldots )$.

By Theorem 7.2,
the system $\left\{ y_{0},\ y_{2},\ y_{3},\ldots\right\}$
with removed $y_{1}$ is a basis  in $L_p(0,1)$ $(1<p<+\infty)$ for any choice of $C$. Similarly, by Theorem 7.4,
the system $\left\{ y_{0},\ y_{1},\ y_{3},\ldots\right\}$
with removed $y_{2}=\cos{\frac{\pi x}{2}}$ is a basis  in $L_p(0,1)$ $(1<p<+\infty)$ for any choice of $C$.

We will consider another 2 special problems, this time with triple eigenvalues. The first one corresponds to the case $\lambda _{k}\neq-\frac{d}{c}$ and the second one to $\lambda _{k}=-\frac{d}{c}$.

\subsection*{Example 2.}Let us take the problem  
$$
-y^{\prime \prime }=\lambda y,\ 0<x<1,
$$
$$
y'(0)=0,\ 
3{\lambda}y(1)=\left( {\lambda }-3\right)y^{\prime }(1).
$$
Note that $a=3$, $b=0$, $c=1$, $d=-3$, $ad-bc=-9<0$, $q(x)\equiv0$,  $\lambda _{0}=\lambda _{1}=\lambda _{2}=0$ is the eigenvalue of multiplicity three, and $\lambda _{0}\neq-\frac{d}{c}$. The remaining eigenvalues $\lambda _{3}<\lambda _{4}<\ldots$ are positive. The eigenfunctions are $y_{0}=1$, $y_{n}=\cos{\sqrt{\lambda _{n}}x}$ $(n\ge 3)$, and by (2.10)-(2.12), the first associated function of $y_{0}$ is $y_{1}=-\frac{1}{2}x^2+C$, where $C$ is
a constant. Furthermore, by (4.1)-(4.3), the second associated function of $y_{0}$ is $y_{2}=\frac{1}{24}x^4-\frac{C}{2}x^2+D$.
By formula (4.9),
$$
Q_{0}=(y_{2},y_{0})-9\left(\frac{y_{2}(1)}{-3} - \frac{y_{1}(1)}{(-3)^{2}}
+\frac{y_{0}(1)}{(-3)^{3}}\right)\cdot\frac{y_{0}(1)}{ -3}=\frac{1}{45}. 
$$
Alternatively, by (4.10),
$$
Q_{0}=\|y_{1}\|^2_{2}-9\left(\frac{y_{1}(1)}{-3} - \frac{y_{0}(1)}{(-3)^{2}}\right)^2
=\frac{1}{45}. 
$$
By (5.8),
$$L_{0}=(y_{1},y_{2})-9\left(\frac{y_{1}(1)}{-3} - \frac{y_{0}(1)}{(-3)^{2}}
\right)\cdot\left(\frac{y_{2}(1)}{-3} - \frac{y_{1}(1)}{(-3)^{2}}
+\frac{y_{0}(1)}{(-3)^{3}}\right)=\frac{2C}{45}-\frac{1}{189}.$$
Then $C_{2}=-\frac{L_{0}}{Q_{0}}=\frac{5}{21}-2C$. Therefore, 
$$
y_{1}^{\#}=y_{1}+C_{2}y_{0}=-\frac{1}{2}x^2+\frac{5}{21}-C .
$$
By (7.5),
$$
\mathfrak{A}(y^{\#}_{1})=\frac{y_{1}^{\#}(1)}{-3} - \frac{y_{0}(1)}{(-3)^{2}}=-\frac{1}{42}+\frac{C}{3}.
$$
Consequently, by Theorem 8.2,
the system $\left\{ y_{0},\ y_{2},\ y_{3},\ldots\right\}$,  with removed $y_{1}$ is a basis  in $L_p(0,1)$ $(1<p<+\infty)$ if and only if $C\ne\frac{1}{14}$.
If $C= \frac{1}{14}$, then $y_{1}^{\#}(x)=-\frac{x^{2}}{2}+\frac{7}{42}$ is orthogonal to all the elements of the system $\left\{ y_{n}\right\}  \ (n=0,\ 2,\ 3,\ldots )$.
Next we calculate 
$$
y_{2}^{*}=y_{2}+C_{2}y_{1}=\frac{1}{24} x^{4}+\frac{1}{2} C \,x^{2}+D-\frac{5}{42} x^{2}+\frac{5}{21} C-2 C^{2},
$$
and by (6.9), 
$$J_{0}=(y_{2}^{*},y_{2})-9\left(\frac{y_{2}^{*}(1)}{-3} - \frac{y_{1}^{\#}(1)}{(-3)^{2}}+\frac{y_{0}(1)}{(-3)^{3}}
\right)\times
$$
$$
\times\left(\frac{y_{2}(1)}{-3} - \frac{y_{1}(1)}{(-3)^{2}}
+\frac{y_{0}(1)}{(-3)^{3}}\right)=-\frac{5}{3969}+\frac{2}{45} D+\frac{2}{189} C-\frac{1}{15} C^{2}.
$$
Then
$$
D_{1}=-\frac{J_{0}}{Q_{0}}=\frac{25}{441}-2 D-\frac{10}{21} C+3 C^{2},
$$
and
$$
y_{2}^{\#}=y_{2}^*+ D_{1}y_{0}=\frac{1}{24} x^{4}+\frac{1}{2} C \,x^{2}-D-\frac{5}{42} x^{2}-\frac{5}{21} C+C^{2}+\frac{25}{441}.
$$
Therefore, 
$$
\mathfrak{A}(y^{\#}_{2})=\frac{y_{2}^{\#}(1)}{-3} - \frac{y_{1}^{\#}(1)}{(-3)^{2}}+\frac{y_{0}(1)}{(-3)^3}
=-\frac{11}{10584}+\frac{1}{42} C+\frac{1}{3} D-\frac{1}{3} C^{2}.
$$
Consequently, by Theorem 8.3,
the system $\left\{ y_{1},\ y_{2},\ y_{3},\ldots\right\}$
with removed $y_{0}$ is a basis  in $L_p(0,1)$ $(1<p<+\infty)$ if and only if $D\ne\frac{11}{3528}-\frac{1}{14} C+C^{2}$.
If $D= \frac{11}{3528}-\frac{1}{14} C+C^{2}$, then $$y_{2}^{\#}(x)=\frac{1}{24} x^{4}+\frac{1}{2} C \,x^{2}+\frac{3}{56}-\frac{1}{6} C-\frac{5}{42} x^{2},$$is orthogonal to all the elements of the system $\left\{ y_{n}\right\}  \ (n=1,\ 2,\ 3,\ldots )$.

By Theorem 8.1,
the system $\left\{ y_{0},\ y_{1},\ y_{3},\ldots\right\}$
with removed $y_{2}$ is a basis  in $L_p(0,1)$ $(1<p<+\infty)$ for any choice of $C$. Similarly, by Theorem 8.4,
the system $\left\{ y_{0},\ y_{1},\ y_{2},\ldots\right\}(n\neq l,\ l\ge3)$
with removed $y_{l}$ is a basis  in $L_p(0,1)$ $(1<p<+\infty)$ for any choice of $C$ and $D$.

\subsection*{Example 3.}Let us now take the problem  
$$
-y^{\prime \prime }=\lambda y,\ 0<x<1,
$$
$$
y(0)=-y'(0),\ 
(9{\lambda}+15)y(1)= 5{\lambda }y^{\prime }(1).
$$
Note that now $a=9$, $b=15$, $c=5$, $d=0$, $ad-bc=-75<0$, again $q(x)\equiv0$,  $\lambda _{0}=\lambda _{1}=\lambda _{2}=0$ is the eigenvalue of multiplicity three, and $\lambda _{0}=-\frac{d}{c}$. The other eigenvalues $\lambda _{3}<\lambda _{4}<\ldots$ are positive. The eigenfunctions are $y_{0}=1-x$, $y_{n}=\cos{\sqrt{\lambda _{n}}x}-\frac{\sin{\sqrt{\lambda_n}x}}{{\sqrt{\lambda_n}}}$ $(n\ge 3)$, and the associated functions are $y_{1}=\frac{x^{3}}{6}-\frac{x^{2}}{2}+C\cdot \left(1-x\right)$ and $y_{2}=-\frac{x^{5}}{120}+\frac{x^{4}}{24}+C\cdot \left(\frac{x^{3}}{6}-\frac{x^{2}}{2}\right)+{D}\cdot \left(1-x\right)$.
By formula (4.11) or (4.12),
$P_{0}=\frac{4}{175}.$
By (5.10), $M_{0}=-\frac{277}{70875}+\frac{8 C}{175}.$
Then $C_{2}=-\frac{M_{0}}{P_{0}}=\frac{277}{1620}-2 C$. Therefore, 
$$
y_{1}^{\#}=\frac{1}{6} x^{3}-\frac{1}{2} x^{2}+C x-C+\frac{277}{1620}-\frac{277}{1620} x.
$$
By (7.5),
$
\mathfrak{A}(y^{\#}_{1})=-\frac{23}{4860}+\frac{C}{15}.
$
Consequently, by Theorem 8.2,
the system $\left\{ y_{0},\ y_{2},\ y_{3},\ldots\right\}$, is a basis in $L_p(0,1)$ $(1<p<+\infty)$ if and only if $C\ne\frac{23}{324}$.
If $C= \frac{23}{324}$, then $y_{1}^{\#}(x)=\frac{1}{6} x^{3}-\frac{1}{2} x^{2}-\frac{1}{10} x+\frac{1}{10}$ is orthogonal to all the elements of the system $\left\{ y_{n}\right\}  \ (n=0,\ 2,\ 3,\ldots )$.
Furthermore, by (6.11), 
$$H_{0}=-\frac{12}{175} C^{2}+\frac{8}{175} {D}+\frac{554}{70875} C-\frac{491767}{252598500}.
$$
Then
$$
D_{1}=-\frac{H_{0}}{P_{0}}= 3 C^{2}-2 {D} -\frac{277}{810} C +\frac{491767}{5773680},
$$
Therefore, 
$$
\mathfrak{A}(y^{\#}_{2})=-\frac{3551}{17321040}+\frac{23}{4860} C+\frac{1}{15} {D}-\frac{1}{15} C^{2}.
$$
Consequently, by Theorem 8.3,
the system $\left\{ y_{1},\ y_{2},\ y_{3},\ldots\right\}$ is a basis  in $L_p(0,1)$ $(1<p<+\infty)$ if and only if $D\ne\frac{3551}{1154736}-\frac{23}{324} C+C^{2}$.  As in the previous example, the basis properties for the other choices of the removed function does not depend on the choice of $C$ and $D$.

\subsection*{Acknowledgment}
This work was financially supported by ADA University and Baku State University.

\section{Declarations}
\textbf{Ethical Approval.}
Not applicable.
 \newline \textbf{Competing interests.}
None.
  \newline \textbf{Authors' contributions.} 
Both authors contributed to the paper equally.
  \newline \textbf{Funding.}
This work was completed with the support of ADA University Faculty Research and Development Fund and Baku State University.
  \newline \textbf{Availability of data and materials.}
The data will be made available upon request.

\end{document}